\documentclass[EJP]{ejpecp}
\usepackage{amsfonts} 

\usepackage{mathrsfs}

\AUTHORS{%
  Chang-Long~Yao\footnote{Academy of Mathematics and Systems Science, CAS,
Beijing, China.
    \EMAIL{deducemath@126.com}}
  }

\SHORTTITLE{Winding angles of the arms}

\TITLE{A CLT for winding angles of the arms for critical planar percolation} 

\KEYWORDS{critical percolation; incipient infinite cluster; winding
angle; central limit
theorem; martingale; arm events} 

\AMSSUBJ{60K35} 
\AMSSUBJSECONDARY{82B43} 

\SUBMITTED{August 30, 2012} 
\ACCEPTED{September 18, 2013} 

\VOLUME{18} \YEAR{2013} \PAPERNUM{85} \DOI{v18-2285}

\ABSTRACT{Consider critical percolation in two dimensions. Under the
condition that there are $k$ disjoint alternating black and white
arms crossing the annulus $A(\ell,n)$, we prove a central limit
theorem and variance estimates for the winding angles of the arms
(as $n\rightarrow \infty,\ell$ fixed). This result confirms a
prediction of Beffara and Nolin (Ann. Probab. 39: 1286--1304, 2011).
Using this theorem, we also get a CLT for the multiple-armed
incipient infinite cluster (IIC) measures.}

\begin{document}



\section{Introduction}
Percolation is a central model of statistical physics.  Recall that
performing a site percolation with parameter $p$ on a lattice means
that each site is chosen independently to be black (open) with
probability $p$ and white (closed) with probability $1-p$.  It is
well-known that site percolation on the regular triangular lattice
exhibits a phase transition at a critical point $p_c=1/2$: when
$p\leq p_c$ there is almost surely no infinite black connected
component, whereas when $p>p_c$ there is almost surely a unique
infinite black connected component.

Consider percolation on a planar lattice.  In the literature, given
an annulus in the lattice, the arm events are referred to the
existence of some number of disjoint paths (arms, see below for a
formal definition) crossing the annulus, the color of each path
(black or white) being prescribed.  These events are very useful for
studying critical and near-critical percolation, the so-called arm
exponents can be used to describe some fractal properties of
critical percolation (see \cite{r19,r17,r18}).

In this paper, we investigate the winding angles of the arms. The
motivation mainly came from \cite{r3}. In that paper, Beffara and
Nolin proved the existence of the monochromatic exponents, and the
monochromatic $j$-arm exponent is strictly between the polychromatic
$j$-arm and $(j+1)$-arm exponents. Their proof relied on analyzing
the winding angles of the monochromatic and polychromatic arms. They
believed that a central limit theorem should hold on the winding
angles but did not give the proof.

The winding angles are also interesting in their own rights. In
fact, the winding angles for several different planar models have
been studied in the literature (e.g., random walk \cite{r11},
Brownian motion \cite{r23,r10}, loop-erased random walk (LERW)
\cite{r14}, self-avoiding walk (SAW) \cite{r12}, radial
Schramm-Loewner evolution (SLE) \cite{r26,r9,r13}, see also Remark
\ref{rem1}).  In these models, from a macroscopic view, the
conformal invariance properties were extensively used to derive the
winding angle variance and CLT.

We focus on site percolation on the triangular lattice at
criticality. We will realize the triangular lattice with site (or
vertex) set $\mathbb{Z}^2$.  For a given $(x,y)\in\mathbb{Z}^2$, its
neighbors are defined as $(x\pm 1,y),(x,y\pm1),(x+1,y-1)$ and
$(x-1,y+1)$. Edges (bonds) between neighboring or adjacent sites
therefore correspond to vertical or horizontal displacements of one
unit, or diagonal displacements between two nearest sites along a
line making an angle of $135^{\circ}$ with the positive $x$-axis.
Let each site of $\mathbb{Z}^2$ be black or white with probability
$1/2$ independently of each other, and denote $P=P_{1/2}$ the
corresponding product probability measure on the set of
configurations.  We also represent the measure as a (black or white)
random coloring of the faces of the dual hexagonal lattice. Let us
mention that the results in this paper also hold for critical bond
percolation on $\mathbb{Z}^2$.

A path is a sequence of distinct sites connected by nearest neighbor
bonds.  The event that two sets of sites $X_1,X_2\subset
\mathbb{Z}^2$ are connected by a black path is denoted by
$X_1\leftrightarrow X_2$, and $X_1,X_2$ are connected by a black or
white path is denoted by $X_1\leftrightarrow_1 X_2$. Given a set $X$
of sites, let $\partial X$ denote the boundary of $X$ which contains
sites in $X$ that are adjacent to some site not in $X$.  A circuit
is a path which starts and ends at the same site and does not visit
the same site twice, except for the starting site.  For a circuit
$\mathcal {C}$, define
$$\overline{\mathcal {C}}:=\mathcal {C}\cup \mbox{ interior sites of }\mathcal {C}.$$
Let $\sigma=(\sigma_i)$ be a sequence of colors.  Given two circuits
$\mathcal {C}_1,\mathcal {C}_2$ such that $\overline{\mathcal
{C}_1}\subset \overline{\mathcal {C}_2}$, we say that $\mathcal
{C}_1$ is $\sigma$-connected to $\mathcal {C}_2$,  if there exist
$|\sigma|$ disjoint paths (\textbf{arms}) connecting $\mathcal
{C}_1$ and $\mathcal {C}_2$, ordered counterclockwise in a cyclic
way, and the color of the $i$-th path is $\sigma_i$. Denote this
event by $\mathcal {C}_1\leftrightarrow_{\sigma}\mathcal {C}_2$.

Define $||\textbf{x}||_{\infty}:=\max\{|x|,|y|\}$ for
$\textbf{x}=(x,y)\in \mathbb{Z}^2$.  For any $r\geq 0$, define the
square box of sites $B(r):=\{\textbf{x}\in
\mathbb{Z}^2:||\textbf{x}||_{\infty}\leq r\}$. For $0<n<m$, define
the annulus $$A(n,m):=B(m)\backslash B(n).$$  For a crossing arm
$\gamma$ in an annulus $A(n,m)$, we often consider $\gamma$ as a
continuous curve by connecting the neighbor sites with line segments
and assume the direction of $\gamma$ is from $\partial B(n)$ to
$\partial B(m)$. The \textbf{winding angle} of $\gamma$ is the
overall (algebraic) variation of the argument along it and is
denoted by $\theta(\gamma)$.

For a polychromatic configuration in the annulus $A(n,m)$ (i.e.,
with at least one arm of each color), it is easy to see that the
winding angles of the arms differ by at most $2\pi$.  In the
following, we fix a deterministic way to choose a unique arm
$\gamma_{n,m}$ and focus on $\theta(\gamma_{n,m})$,  since there is
essentially a unique winding angle from a macroscopic point of view.

For two positive functions $f$ and $g$, the notation $f\asymp g$
means that $f$ and $g$ remain of the same order of magnitude, in
other words that there exist two positive and finite constants $c_1$
and $c_2$ such that $c_1g\leq f\leq c_2 g$.

Now we give our main result in Theorem \ref{t1}, from which we see
that a crossing arm of a polychromatic configuration in a long
annulus looks like a random logarithmic spiral.
\begin{theorem}\label{t1}
Assume that $\sigma$ is alternating and $|\sigma|$ is even.  Let $l$
be the minimal number such that $|\partial B(l)|\geq |\sigma|$.  We
condition on the event $\partial B(l)\leftrightarrow_{\sigma}
\partial B(n),n>l$.  Let $\theta_n:=\theta(\gamma_{l,n})$, and $a_n:=\sqrt{Var[\theta_n]}$. Then
we have
\begin{equation*}
a_n\asymp \sqrt{\log n},~~n>l,
\end{equation*}
and under the conditional measure $P(\cdot|\partial
B(l)\leftrightarrow_{\sigma}
\partial B(n))$
\begin{equation*}
\frac{\theta_n}{a_n}\rightarrow_d N(0,1).
\end{equation*}
\end{theorem}
\begin{remark}
Following the conjecture made by Wieland and Wilson \cite{r13} (see
(\ref{e44}) below), it is expected that
$$a_n^2=\left(\frac{6}{|\sigma|^2}+o(1)\right)\log n~~\mbox{as }n\rightarrow \infty,$$ which might be proved by
conformal invariance and SLE approach.  Heuristically, one can
decompose a typical arm into a short path near the origin and for
which the winding angle contribution is of a smaller order than
$\sqrt{\log n}$ and a long path far from the origin and for which
the winding angle contribution can be approximated by the winding
angle of multiple (mutually-avoiding) SLE paths (for multiple SLE
paths, see Remark \ref{rem1}).  However, it is still not clear how
to prove it rigorously.  We will actually use another sequence
$h_n\sim a_n$ instead of $a_n$, for the expressions for $h_n$, see
(\ref{e43}).
\end{remark}
\begin{remark}
Our proof mainly relies on the Strong Separation Lemma and the
coupling argument in \cite{r4}.  These two ingredients may be
extended to the following more general case without too much work:
$\sigma$ is polychromatic and $\sigma$ either does not contain
neighboring white colors or does not contain neighboring black
colors (here we take the first and last elements of $\sigma$ to be
neighbors).  See Remark 7 in \cite{r2} and subsection 5.4 in
\cite{r4}. Thus Theorem \ref{t1} can also be extended to this case.
\end{remark}
\begin{remark}\label{rem1}
There exist some analogous results on the winding angles of various
random paths.  For the classic results on random walk and Brownian
motion, the interested reader is referred to a short survey
\cite{r27}.  We address some results concerning SLE as follows.  For
radial $SLE_{\kappa}$, Schramm \cite{r9} showed that the variance of
the winding angle of the radial $SLE_{\kappa}$ path truncated at
distance $\varepsilon$ from the origin grows like
$(\kappa+o(1))\log(1/\varepsilon)$ (see also \cite{r22}), a CLT was
proved simultaneously.  However, as the authors said in \cite{r13},
conditioned there are $k$ disjoint random paths in a long annulus,
there are few results about the windings compared with the one path
case. Conditioned on the event that there are $k$ mutually-avoiding
$SLE_{\kappa}$ paths crossing the annulus $A(1,R)$ of
$\mathbb{R}^2$, Wieland and Wilson \cite{r13} made the conjecture
that the winding angle variance of the paths is
\begin{equation}\label{e44}
\left(\frac{\kappa}{k^2}+o(1)\right)\log R~~\mbox{as
}R\rightarrow\infty.
\end{equation}
 In \cite{r14}, conditioned on the annulus $A(1,R)$ of
$\delta\mathbb{Z}^2$ has $2$ (resp. $3$) disjoint LERWs, Kenyon
showed that the winding angle of the paths has variance tending to
$(\frac{1}{2}+o(1))\log R$ (resp. $(\frac{2}{9}+o(1))\log R$) as
$R\rightarrow\infty$ while $\delta\rightarrow 0$ (see ``Remarks on
LERW'' in \cite{r13} about Kenyon's incorrect values). This confirms
the formula (\ref{e44}) in the cases of $\kappa=2$ and $k=2,3$,
since LERW converges to $SLE_{2}$ \cite{r28}.  We also note that in
section 8 and subsection 10.6 in \cite{r20}, using the method from
quantum gravity, Duplantier showed the formula \ref{e44}.  See also
\cite{r26} for the proof from Coulomb gas method.
\end{remark}

\emph{Idea of the proof.}
\begin{figure}
\begin{center}
\includegraphics[height=0.4\textwidth]{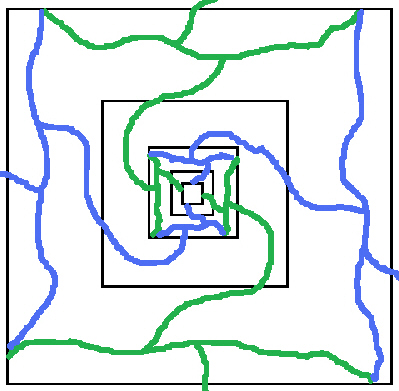}
\caption{We construct a sequence of Markovian good faces to get a
martingale structure of the winding angle.}\label{fig3}
\end{center}
\end{figure}
The strategy of the proof is similar to \cite{r1}. In that paper,
Kesten and Zhang constructed a sequence of black circuits
surrounding the origin in a Markovian way (the circuits could be
thought of as stopping times).  Using these circuits, they got a
martingale structure on the maximal number of disjoint black
circuits in a large box, and then they applied McLeish's CLT
\cite{r7} for the martingales. However, for our setting clearly we
can not use Markovian black circuits to get a martingale structure
of the winding angle, thanks to \cite{r4}, we can use faces instead.
As introduced in Section 3 of \cite{r4}, the faces are some type of
circuits which are composed of alternating color paths. With some
conditions added to the faces, we construct a sequence of good faces
to get a martingale structure of the winding angle. See Fig.
\ref{fig3}. Since we are considering conditional measure, it is hard
to estimate some events and check the conditions in McLeish's CLT.
Thanks to the coupling argument in \cite{r4}, we can get some weak
dependence of the faces and carry out the method from \cite{r1}. In
the proof of Theorem 1.6 in \cite{r2}, the authors used black
circuits with defects, we note that these circuits may not adapt to
the proof of our setting.

Let us give a direct corollary of Theorem \ref{t1} in the following.
First we introduce some definitions.  In the celebrated paper
\cite{r5}, Kesten gave the mathematical rigorous definition of the
incipient infinite cluster, which describes large critical
percolation clusters from the microscopic (lattice scale)
perspective \cite{r15} and the configuration at a ``typical
exceptional time'' of dynamical critical percolation \cite{r25}.
More precisely, let $\textbf{0}$ denote the origin, it is shown in
\cite{r5} that the limit
$$\nu(E):=\lim_{n\rightarrow\infty}P(E|\textbf{0}\leftrightarrow\partial B(n))$$
exists for any event $E$ that depends on the state of finitely many
sites in $\mathbb{Z}^2$.  The unique extension of $\nu$ to a
probability measure on configurations of $\mathbb{Z}^2$ exists and
we call $\nu$ the \textbf{incipient infinite cluster (IIC) measure}
or one-armed IIC measure.  Following Kesten's spirit, Damron and
Sapozhnikov introduced multiple-armed IIC measures in \cite{r2}.
Suppose that $\sigma$ is alternating and let $\ell$ be the minimal
number such that $|\partial B(l)|\geq |\sigma|$.  For every cylinder
event $E$, it is shown in Theorem 1.6 in \cite{r2} the limit
\begin{equation}\label{e20}
\nu_{\sigma}(E):=\lim_{n\rightarrow\infty}P(E|\partial
B(l)\leftrightarrow_{\sigma}\partial B(n))
\end{equation}
exists.  The unique extension of $\nu_{\sigma}$ to a probability
measure on the configurations of $\mathbb{Z}^2$ exists. We call
$\nu_{\sigma}$ the \textbf{${\sigma}$-IIC measure}.

\begin{corollary}\label{c1}
Suppose that $\sigma$ is alternating and let $l$ be the minimal
number such that $|\partial B(l)|\geq |\sigma|$. Suppose $a_n,n>l$
is the sequence defined in Theorem \ref{t1}.  Under
$P(\cdot|\partial B(l)\leftrightarrow \partial B(m)),m>n$ and
$\nu_{\sigma}$,  we define $\theta_n$ similarly as in Theorem
\ref{t1}.  Under the above measures, we have
\begin{equation*}
\frac{\theta_n}{a_n}\rightarrow_d N(0,1).
\end{equation*}
\end{corollary}

\begin{remark}
Under $P(\cdot|\partial B(l)\leftrightarrow \partial B(m))$ and
$\nu_{\sigma}$, using Lemma \ref{l1} and the coupling result in
Lemma \ref{l2}, it is not hard to check that
$\sqrt{Var(\theta_n)}=(1+o(1))a_n$.  We shall not give the proof
here, though.
\end{remark}

\begin{remark}
Following the spirit of \cite{r15,r13}, choosing a typical site from
the boundary (or external perimeter) of a large cluster or the
pivotal sites of a crossing event in a large box uniformly, we can
consider how the arms wind around the chosen site. Since it is
expected that the local measure viewed from the typical site
converges to the corresponding ${\sigma}$-IIC measure, one would
expect a CLT from this corollary. (For the existence of the limiting
measures, see Remark of Theorem 1 in \cite{r15}.  In the 4-arm case,
see also Remark 1.7 in \cite{r24} for the analog of the tightness
result in Theorem 8 in \cite{r15}.)
\end{remark}

For a monochromatic $\sigma$, there are many ways to select the arms
and the winding angles of these arms may differ a lot. Let
$j=|\sigma|$.  We denote by $I_{j,n}$ the set of all the winding
angles of the arms in the annulus $A(\ell,n)$, where $|\partial
B(l)|\geq j$.  Let
$\theta_{max,n}=\theta_{max,j,n}:=\max\{\alpha,\alpha\in I_{j,n}\}$
and $\theta_{min,n}:=\min\{\alpha,\alpha\in I_{j,n}\}$.  It is easy
to show (see \cite{r3}) that $\theta_{max,n}$ and $\theta_{min,n}$
are of order $\pm\log n$. Furthermore, by Proposition 7 in
\cite{r3}, if one sorts the elements of $I_{j,n}$ in increasing
order: $\alpha_1<\alpha_2<\cdots<\alpha_{|I_{j,n}|}$, then for every
$1\leq i\leq |I_{j,n}|-1$, $\alpha_{i+1}-\alpha_i<2\pi$.

In the 1-arm case, one can also get central limit theorems for
$\theta_{max,n}$ and $\theta_{max,n}-\theta_{min,n}$ by similar
methods for the proof of Theorem \ref{t1}.  Using good black
circuits and the coupling argument for the 1-arm case in \cite{r4},
the proof is similar and simpler, we leave it to the reader and just
give the following statements for $\theta_{max,n}$.

Under the conditional measure $P(\cdot|
\textbf{0}\leftrightarrow_{\sigma} \partial B(n))$ and the IIC
measure $\nu$ we both have
\begin{equation*}
E[\theta_{max,n}]\asymp\log n,~~Var[\theta_{max,n}]\asymp \log n,
\end{equation*}
and
\begin{equation*}
\frac{\theta_{max,n}-E[\theta_{max,n}]}{\sqrt{Var[\theta_{max,n}]}}\rightarrow_d
N(0,1).
\end{equation*}
In this paper, we only prove the alternating four arm case, since
the proof for this case applies to all cases that $\sigma$ is
alternating, with no essential changes. In general, we assume
$\sigma=(black,white,black,white)$ in the following.

Throughout this paper, $c,c_{1},c_{2},\ldots$ denote positive finite
constants that may change from line to line or page to page
according to the context.

\section{Preliminary results}\label{s1}
As remarked above, we focus on the alternating four arm case.
Firstly, following the terminology of \cite{r4}, let us introduce
some definitions. Suppose $\Gamma$ is the set of percolation
interfaces which cross the annulus $A=A(m,n)$. If there are $p\geq
2$ interfaces crossing $A$ and if $x_1,\ldots,x_p$ denote the
endpoints of these interfaces on $\partial B(n)$, define the
\textbf{quality}
$$Q(\Gamma):=\frac{1}{n}\inf_{k\neq l}|x_k-x_l|,$$
where $|\cdot|$ denotes Euclidean distance. If $\Gamma=\emptyset$,
we define $Q(\Gamma)=0$.

Let $x_1,\ldots,x_4$ be four midpoints of four distinct bonds in
$\partial B(n)$.  We will adopt here cyclic notation, i.e., for any
$i,j\in \mathbb{Z}$, we have $x_j=x_i$ if $j\equiv i \mod 4$.  For
any $i\in \mathbb{Z}$, let $\gamma_i$ be a simple path of hexagons
joining $x_i$ to $x_{i+1}$ and $\gamma_i\subset B(n)$ (here we see
$\gamma$ as a sequence of sites).  Assume $\gamma_i$ is black if $i$
is odd and white otherwise.  Then we call the circuit $\Theta$ which
is composed of these four paths a configuration of
\textbf{(interior) faces}, and say $\Theta$ are faces of $\partial
B(n)$.  Define the quality of a configuration of faces $Q(\Theta)$
to be the least distance between the endpoints (i.e.,
$x_1,\ldots,x_4$) , normalized by $n$.  Similar to the definition of
faces, we call a circuit around $\partial B(n)$ \textbf{exterior
faces}, if the circuit is composed of four alternating color paths
contained in $(\mathbb{Z}^2\backslash B(n))\cup \partial B(n)$ with
endpoints on $\partial B(n)$.  Note that our definition of exterior
faces is exactly the same as the definition of faces in \cite{r4}.
Similar to the quality of faces, we can also define quality of
exterior faces.

The following properties of arm events are well-known, see
\cite{r21,r17}. We assume that the reader is familiar with the
FKG-inequality, the Russo-Seymour-Welsh (RSW) technology.  See
\cite{r8,r18}.  Using FKG, RSW and Theorem 11 in \cite{r17}, the
statements related to faces can be easily obtained from the classic
results, the proof is omitted here.  Note that for general
alternating color sequence $\sigma$ with even $|\sigma|$, the
corresponding notion of faces can be defined, and analogous results
hold in this more general case.
\begin{enumerate}
\item A priori bound for arm events: For any color sequence $\sigma$, there
exist constants $c(|\sigma|),\beta(|\sigma|)>0$ such that for all
$n_1<n_2$,
\begin{equation}\label{e40}
P(\partial B(n_1)\leftrightarrow_{\sigma}\partial B(n_2))\geq
c\left(\frac{n_1}{n_2}\right)^{\beta}.
\end{equation}
Furthermore, given $\sigma=(black,white,black,white)$, for any faces
$\Theta$ of $\partial B(n_1)$ with $Q(\Theta)>\frac{1}{4}$,
\begin{equation}\label{e41}
P(\Theta\leftrightarrow_{\sigma}\partial B(n_2))\geq
c\left(\frac{n_1}{n_2}\right)^{\beta}.
\end{equation}
\item \textbf{Quasi-multiplicativity}: For any color sequence $\sigma$, there
is a constant $c(|\sigma|)>0$, such that for all $n_1<n_2<n_3$,
\begin{eqnarray*}
&&cP(\partial B(n_1)\leftrightarrow_{\sigma}\partial
B(n_2))P(\partial B(n_2)\leftrightarrow_{\sigma}\partial B(n_3))\leq
P(\partial B(n_1)\leftrightarrow_{\sigma}\partial B(n_3))\\
&&~~~~~~~~~~~~~~~~~~~~~~~~~~~~~~~~~~~~~~\leq P(\partial
B(n_1)\leftrightarrow_{\sigma}\partial B(n_2))P(\partial
B(n_2)\leftrightarrow_{\sigma}\partial B(n_3)).
\end{eqnarray*}
Furthermore, given $\sigma=(black,white,black,white)$, for any faces
$\Theta$ of $\partial B(n_1)$ with $Q(\Theta)>\frac{1}{4}$,
\begin{equation*}
P(\Theta\leftrightarrow_{\sigma}\partial B(n_2))P(\partial
B(n_2)\leftrightarrow_{\sigma}\partial B(n_3))\asymp
P(\Theta\leftrightarrow_{\sigma}\partial B(n_3)).
\end{equation*}
\end{enumerate}

Define $R(m,n):=\{z\in\mathbb{Z}^2:|\arg(z)|<\frac{\pi}{10}\}\cap
A(m,n)$. We say a path $\gamma\subset R(m,n)$ is a crossing path in
$R(m,n)$ if the endpoints of $\gamma$ lie adjacent (Euclidean
distance smaller than $\sqrt{2}$) to the rays of argument
$\pm\frac{\pi}{10}$ respectively. By step 3 of the proof of Theorem
5 in \cite{r3}, we obtain the following lemma.
\begin{lemma}\label{l1}
 Define event
$$\mathcal {B}:=\{\mbox{there exist at least
$K\log(n/m)$ disjoint black crossing paths in $R(m,n)$}\}.$$  There
exist constants $c_1,c_2>0$, such that for all $K>0$ and $n>m>0$,
\begin{equation*}
P(\mathcal {B})\leq c_1\exp[(-c_2K+c_1\log K)\log(n/m)].
\end{equation*}
In particular, there exist constants $c_3,K_0>0$, such that for all
$K>K_0$ and $n>m>0$,
\begin{equation*}
P(\mathcal {B})\leq c_1\exp[-c_3K\log(n/m)].
\end{equation*}
\end{lemma}

Define $\mathcal
{C}_i:=\{z\in\mathbb{C}:-\frac{3\pi}{8}+\frac{i\pi}{4}<arg(z)<-\frac{\pi}{8}+\frac{i\pi}{4}\},1\leq
i\leq 8$.  Lemma \ref{l3} is the straightforward analog of Lemma 3.4
in \cite{r4} with a little modification. See Fig. \ref{fig1}. The
proof is exactly the same as the second proof of Lemma 3.4 in
\cite{r4}, see Figure 3.2 in \cite{r4} for the strategy.  We will
adopt cyclic notation in Lemma \ref{l3}, i.e., for any $i,j\in
\mathbb{Z}$, we have $\mathcal {C}_j=\mathcal {C}_i$ if $j\equiv i
\mod 8$.
\begin{figure}
\begin{center}
\includegraphics[height=0.4\textwidth]{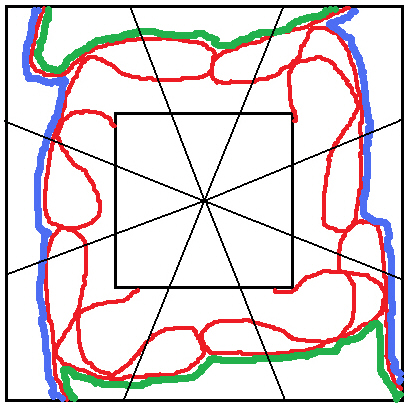}
\caption{Four interfaces crossing the annulus induce a natural
configuration of good faces.}\label{fig1}
\end{center}
\end{figure}
\begin{lemma}\label{l3}
In the annulus $A=A(n,2n)$, let $\mathcal {R}$ be the event that
there are exactly 4 disjoint alternating arms crossing $A$, and the
resulting 4 interfaces are contained respectively in $\mathcal
{C}_{i-1}\cup\mathcal {C}_{i}\cup\mathcal {C}_{i+1},i=2,4,6,8$, with
the endpoints of the interfaces on the two boundaries of $A$
belonging to $\mathcal {C}_i$. Then $P(\mathcal {R})>c$ for an
absolute constant $c>0$.
\end{lemma}

Define
$$A(p):=A(2^{p},2^{p+1}),~~p\geq 0.$$
For each $A(p)$, if the event $\mathcal {R}$ (see the definition in
Lemma \ref{l3}) happens, then the four interfaces induce a natural
configuration of faces $\Theta\subset A(p)$ of $\partial
B(2^{p+1})$.  We call $\Theta$ \textbf{good faces}, and say there
are good faces in $A(p)$. See Fig. \ref{fig1}. $\Theta$ are composed
of four paths $\{\Theta(j)\}_{1\leq j\leq 4}$, where the path
$\Theta(j)\subset
\{z\in\mathbb{C}:(j-2)\frac{\pi}{2}-\frac{\pi}{8}<arg(z)<j\frac{\pi}{2}+\frac{\pi}{8}\}$.

\begin{lemma}\label{l2}
There is a constant $c_1>0$, such that for all $r=2^{p_0},R>10r$,
$0\leq t\leq \log_2(R/r)$ and any faces $\Theta$ of $\partial B(r)$,
\begin{equation}\label{e29}
P\left(\bigcap_{0\leq i\leq t}\{\mbox{there are no good faces in
}A(p_0+i)\}|\Theta\leftrightarrow_{\sigma}
\partial B(R)\right)\leq \exp(-c_1t).
\end{equation}
Furthermore, there exists a constant $c_2>0$, such that for any
$r_1=2^{p_1},r_2=2^{p_2}$ with $r_1\leq r_2<R/10$, given any faces
$\Theta_1$ of $\partial B(r_1)$ and faces $\Theta_2$ of $\partial
B(r_2)$, for any $0\leq t\leq \log_2(R/r_2)$, there is a
\textbf{coupling} of the measures
$P(\cdot|\Theta_1\leftrightarrow_{\sigma}\partial B(R))(\mbox{or }
P(\cdot|\partial B(1)\leftrightarrow_{\sigma}\partial B(R)))$ and
$P(\cdot|\Theta_2\leftrightarrow_{\sigma}\partial B(R))$ so that
with probability at least $1-\exp(-c_2t)$, there exists $0\leq i\leq
t$ such that $A(p_2+i)$ has identical good faces $\Theta_0$ for both
measures, and the configurations in the domain $B(R)\backslash
\overline{\Theta}_0$ are also identical.
\end{lemma}

\begin{proof}
Since the proof is basically the same as for Proposition 3.1 in
\cite{r4} with little modifications for our setting, we only sketch
the proof and omit some details. First let us prove inequality
(\ref{e29}).

Without loss of generality, let $N=\lfloor t/8\rfloor\geq 1$. For
$0\leq i\leq N$, let $r_i=8^ir$. Now we sample under
$P(\cdot|\Theta\leftrightarrow \partial B(R))$ the set of interfaces
$\Gamma(r_i)$ which start from the four endpoints of $\Theta$ until
they reach radius $r_i$.  We proceed by induction on the scale
$r_i,i\geq 0$.  By the Strong Separation Lemma (Lemma 3.3 in
\cite{r4}, Lemma 6.2 in \cite{r2} is another version), we have
\begin{equation}\label{e7}
P(Q(\Gamma(2r_i)>1/4)|\Theta\leftrightarrow_{\sigma}\partial
B(R),\Gamma(r_i))>c_3.
\end{equation}
It can be checked that $\Theta$ plus $\Gamma(2r_i)$ induce
configurations of faces of $\partial B(2r_i)$, which is denoted by
$\Theta_i$.  For each $A(4r_i,r_{i+1})$, define
$$\mathcal {R}_i:=\{\mbox{there are good faces in
} A(4r_i,r_{i+1})\}.$$ Let $S_i$ be the union of all sites whose
color is determined by the crossing interfaces according to the
measure $P(\cdot|\Theta_i,\Theta_i\leftrightarrow_{\sigma}\partial
B(R))$. Let $S_i$ be a possible value for $S_i$ such that $\mathcal
{R}_i$ holds.  If $Q(\Gamma(2r_i))>1/4$, by Lemma \ref{l3}, using a
gluing technique, it can be showed (the same to the proof of (3.1)
in \cite{r4}) there is a universal constant $c>0$ such that
\begin{equation}\label{e28}
2^{-|S|}/c\leq
P(S_i=S|\Theta_i,\Theta_i\leftrightarrow_{\sigma}\partial B(R))\leq
c2^{-|S|}.
\end{equation}
Then we get
\begin{equation}\label{e8}
P(\mathcal {R}_i|\Theta_i,\Theta_i\leftrightarrow_{\sigma}\partial
B(R))\geq c_4P(\mathcal {R}_i)>c_5.
\end{equation}
Combining (\ref{e7}) and (\ref{e8}), we have
\begin{equation}\label{e30}
P(\mathcal {R}_i|\Theta\leftrightarrow_{\sigma}\partial
B(R),\Gamma(r_i))>c_3c_5.
\end{equation}
By (\ref{e30}), choosing $c_2$ appropriately, we get
\begin{eqnarray*}
&&P\left(\bigcap_{0\leq i\leq t}\{\mbox{there are no good faces in
}A(p_0+i)\}|\Theta\leftrightarrow_{\sigma}
\partial B(R)\right)\\
&&~~~~\leq P(\mathcal {R}_0^c|\Theta\leftrightarrow_{\sigma}
\partial B(R))P(\mathcal {R}_1^c|\Theta\leftrightarrow_{\sigma}
\partial B(R),\mathcal {R}_0^c)\ldots \\
&&~~~~~~~~~P(\mathcal {R}_{N-1}^c|\Theta\leftrightarrow_{\sigma}
\partial B(R),\mathcal {R}_0^c\mathcal {R}_1^c\ldots\mathcal {R}_{N-2}^c)\\
&&~~~~\leq P(\mathcal {R}_0^c|\Theta\leftrightarrow_{\sigma}
\partial B(R))\max_{\Gamma(r_1)}\{P(\mathcal {R}_1^c|\Theta\leftrightarrow_{\sigma}
\partial B(R),\Gamma(r_1))\}\ldots\\
&&~~~~~~~~~\max_{\Gamma(r_{N-1})}\{P(\mathcal
{R}_{N-1}^c|\Theta\leftrightarrow_{\sigma}
\partial B(R),\Gamma(r_{N-1}))\}\\
&&~~~~\leq \exp(-c_2t).
\end{eqnarray*}
Now let us prove the coupling result. Sampling the interfaces for
the measures $P(\cdot|\Theta_1\leftrightarrow_{\sigma}\partial
B(R))$ and $P(\cdot|\Theta_2\leftrightarrow_{\sigma}\partial B(R))$
by induction similarly to the above argument, using the Strong
Separation Lemma and (\ref{e28}), one can show the coupling result
with the strategy very similar to the proof of Proposition 3.1 in
\cite{r4}. We omitted the details here and refer the reader to
``Proof of Proposition 3.1, continued.'' in \cite{r4}. To couple
$P(\cdot|\partial B(1)\leftrightarrow_{\sigma}\partial B(R))$ and
$P(\cdot|\Theta_2\leftrightarrow_{\sigma}\partial B(R))$, one can
use an argument analogous to the proof of Proposition 3.6 in
\cite{r4}, the details are also omitted here.
\end{proof}
\begin{remark}
The coupling argument was introduced in \cite{r4}, which is a very
useful tool to gain weak independence of events.  The coupling
argument is based on the Strong Separation Lemma, which was first
proposed in \cite{r2} (see a broad overview for the strong
separation phenomenon in many planar statistical physics models in
Appendix A of \cite{r4}), and is an extension of Kesten's arm
separation lemma \cite{r21,r17}.
\end{remark}
Consider measure $P(\cdot|\partial
B(1)\leftrightarrow_{\sigma}\partial B(n))$, $n\geq
2^{q+q^{\frac{1}{3}}+2}$ (the exponent $1/3$ can be replaced with
any fixed positive constant which is smaller than $1/2$). Define for
$1\leq p\leq q$
$$m(p)=m(p,\omega):=\min\{t\in\{p,p+1,\cdots\}:\mbox{there exist good faces in }A(t)\}.$$
Define $$\mathcal {A}_q:=\{m(q)\leq q+q^{\frac{1}{3}}\}.$$ The
following lemma implies that when we go out of a box (or faces) to
search good faces, we can quickly find them with high probability.
\begin{lemma}\label{l9}
There exists constant $c_1>0$ such that for $1\leq p\leq q$, $0\leq
t\leq q+q^{\frac{1}{3}}-p$ and $n\geq 2^{q+q^{\frac{1}{3}}+2}$,
\begin{equation}\label{e18}
P(m(p)-p\geq t|\partial B(1)\leftrightarrow_{\sigma}\partial
B(n))\leq \exp(-c_1t).
\end{equation}
Furthermore, there exists constant $c_2>0$, such that for $1\leq
p_1< p\leq q$, given faces $\Theta_{p_1}$ of $\partial B(2^{p_1})$,
\begin{equation}\label{e25}
P(m(p)-p\geq t|\Theta_{p_1}\leftrightarrow_{\sigma}\partial
B(n))\leq \exp(-c_2t).
\end{equation}
Let $\Theta_0:=\partial B(1).$ In particular, there exists constant
$c_3>0$, such that for $0\leq p\leq q-1$, given faces $\Theta_p$ of
$\partial B(2^{p})$,
\begin{equation}\label{e19}
P(\mathcal {A}_q^c|\Theta_p\leftrightarrow_{\sigma}
\partial B(n))\leq \exp(-c_3q^{\frac{1}{3}}).
\end{equation}
\end{lemma}

\begin{proof}
If there exist 4 alternating arms from $\partial B(1)$ to $\partial
B(2^{p+\lfloor t/2\rfloor})$ and there is no extra disjoint arm, the
crossing interfaces between $\partial B(1)$ to $\partial
B(2^{p+\lfloor t/2\rfloor})$ induce faces $\Theta$ of $\partial
B(2^{p+\lfloor t/2\rfloor})$.  So we look at whether there is an
extra arm or not, and if not, we can condition on faces $\Theta$.
This means
\begin{eqnarray*}
&&P(m(p)-p\geq t|\partial B(1)\leftrightarrow_{\sigma}\partial
B(n))\\
&&\leq \frac{P(\partial B(1)\leftrightarrow_1\partial B(2^{p+\lfloor
t/2\rfloor})\Box
\partial B(1)\leftrightarrow_{\sigma}
\partial B(n))}{P(\partial B(1)\leftrightarrow_{\sigma} \partial
B(n))}\\
&&~~+\sum_{\Theta}P(\Theta|\partial B(1)\leftrightarrow_{\sigma}
\partial B(n))P(m(p+\lfloor t/2\rfloor)-p-\lfloor t/2\rfloor\geq t/2|\Theta,\Theta \leftrightarrow_{\sigma}\partial
B(n)).
\end{eqnarray*}
Using Reimer's inequality \cite{r6}, RSW and Lemma \ref{l2},
choosing $c_1,c_4>0$ appropriately,  we get
\begin{eqnarray*}
&&P(m(p)-p\geq t|\partial B(1)\leftrightarrow_{\sigma}\partial
B(n))\\
&&\leq P(\partial B(1)\leftrightarrow_1\partial B(2^{p+\lfloor
t/2\rfloor})+\sum_{\Theta}P(\Theta|\partial
B(1)\leftrightarrow_{\sigma}
\partial B(n))\exp(-c_4t)\\
&&\leq \exp(-c_1t).
\end{eqnarray*}
The proof of (\ref{e25}) is similar and simpler, we leave it to the
reader.  Applying (\ref{e18}) and (\ref{e25}), we obtain (\ref{e19})
immediately.
\end{proof}

\section{Proof of theorem}\label{s2}

As remarked in the Introduction, we focus on the alternating 4-arm
case throughout this section.  Let $\Theta_p$ denote the good faces
in $A(m(p))$, $1\leq p\leq q$. Recall $\Theta_0=\partial B(1)$.
Define
$$\mathcal {F}_p:=\mbox{$\sigma$-field generated by
}\{0,1\}^{\overline{\Theta}_p\backslash \textbf{0}},~~1\leq p\leq
q.$$

Let $\mathcal {F}_0$ be the trivial $\sigma$-field.

Define

$$P_q(\cdot)=P_{q,n}(\cdot):=P(\cdot|\mathcal {A}_q,\partial B(1)\leftrightarrow_{\sigma}\partial B(n)),$$
$$E_q(\cdot)=E_{q,n}(\cdot):=E(\cdot|\mathcal {A}_q,\partial B(1)\leftrightarrow_{\sigma}\partial B(n)).$$

Let us concentrate on the measure $P_q$ in the following.

For $1\leq p_1<p_2\leq q$, we define the winding angle between good
faces $\Theta_{p_1}$ and $\Theta_{p_2}$ as follows.  If
$\Theta_{p_1}=\Theta_{p_2}$ (note that this happens when $m(p_1)=
m(p_2)$, which is possible), let
$\theta(\Theta_{p_1},\Theta_{p_2})=0$. Now suppose
$\Theta_{p_1}\neq\Theta_{p_2}$ and
$\gamma([0,1])\subset(\overline{\Theta}_{p_2}\backslash\overline{\Theta}_{p_1})\cup\Theta_{p_1}
$ is a black or white path connecting $\Theta_{p_1}$ and
$\Theta_{p_2}$. Recall the definition of $\Theta(j)$ which is
defined after the definition of good faces.  Assume the starting
point $\gamma(0)\in \Theta_{p_1}(k_1)$ and the endpoint
$\gamma(1)\in \Theta_{p_2}(k_2)$, $1\leq k_1,k_2\leq 4$. We connect
$\gamma'(0):=2^{m(p_1)+1}\exp(i(k_1-1)\frac{\pi}{2})$ and
$\gamma(0)$ by the segment $\overline{\gamma'(0)\gamma(0)}$ and
connect $\gamma(1)$ and
$\gamma'(1):=2^{m(p_2)+1}\exp(i(k_2-1)\frac{\pi}{2})$ by
$\overline{\gamma(1)\gamma'(1)}$. Then we construct a new path
$\gamma':=\overline{\gamma'(0)\gamma(0)}\gamma\overline{\gamma(1)\gamma'(1)}$.
\begin{figure}
\begin{center}
\includegraphics[height=0.4\textwidth]{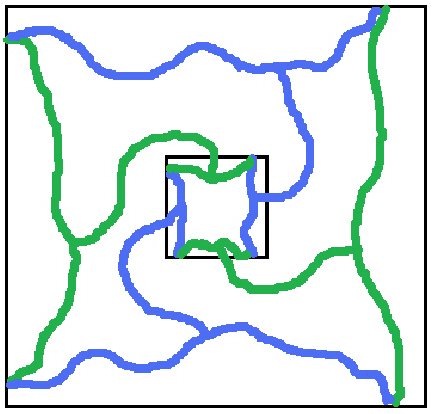}
\caption{Denote the good faces in this figure by $\Theta_1$ and
$\Theta_2$. It is clear that
$\theta(\Theta_1,\Theta_2)=\frac{\pi}{2}$.}\label{fig2}
\end{center}
\end{figure}
Define
$$\theta(\Theta_{p_1},\Theta_{p_2}):=\theta(\gamma'),~~1\leq
p_1<p_2\leq q.$$ See Fig. \ref{fig2}. By the definition of good
faces, it is easy to see that $\theta(\cdot,\cdot)$ is well defined
and independent of the choice of $\gamma$.

Recall $\Theta_0=\partial B(1)$.  For $1\leq p\leq q$, now we define
$\theta(\Theta_0,\Theta_{p})$ similarly.  Among the black or white
paths connecting $\Theta_0$ and $\Theta_{p}$ in
$\overline{\Theta}_{p}\backslash \textbf{0}$, we choose the unique
one in some definite way, and denote it by $\gamma$. Assume the
endpoint $\gamma(1)\in \Theta_{p}(k),1\leq k\leq 4$. Then we connect
$\gamma(1)$ and $\gamma'(1):=2^{m(p)+1}\exp(i(k-1)\frac{\pi}{2})$ by
$\overline{\gamma(1)\gamma'(1)}$.  Let $\gamma'(0)=\gamma(0)$. Then
we construct a new path
$\gamma':=\gamma\overline{\gamma(1)\gamma'(1)}$.  Define
$$\theta(\Theta_0,\Theta_{p}):=\theta(\gamma'),~~1\leq p\leq q.$$

By a simple topological argument, we get
\begin{equation}\label{e21}
\theta(\Theta_0,\Theta_q)=\sum_{i=0}^{q-1}\theta(\Theta_i,\Theta_{i+1}).
\end{equation}
Define for $1\leq p\leq q$
\begin{equation}\label{e22}
\Delta_p=\Delta_{p,q,n}:=E_q(\theta(\Theta_0,\Theta_q)|\mathcal
{F}_p)-E_q(\theta(\Theta_0,\Theta_q)|\mathcal {F}_{p-1}).
\end{equation}
Thus
$$\theta(\Theta_0,\Theta_q)-E_q(\theta(\Theta_0,\Theta_q))=\sum_{p=1}^{q}\Delta_p.$$
See Fig. \ref{fig3} for an illustration.

Similarly to \cite{r1}, we have to become more specific about the
probability space. Set $\Omega=\{black,white\}^{\mathbb{Z}^2}$. The
$\sigma$-field $\mathscr{B}$ is generated by the cylinder sets of
$\Omega$. Recall the notation $P(\cdot)$ we introduced earlier, the
underlying probability is just $(\Omega,\mathscr{B},P)$.  Let
$(\Omega',\mathscr{B}',P')$ be a copy of $(\Omega,\mathscr{B},P)$.
We need some cumbersome but unavoidable expressions in the
following. For example, to determine
\begin{equation}\label{e42}
\theta(\Theta_{p}(\omega),\Theta_q(\omega'))(\omega')
\end{equation}
in Lemma \ref{l4} below one first determines $\Theta_{p}(\omega)$ in
the configuration $\omega\in\Omega$ under
$(\Omega,\mathscr{B},P_q)$, and then $\Theta_q(\omega')$ in the
configuration $\omega'\in\Omega'$ under
$(\Omega',\mathscr{B}',P'(\cdot|\mathcal
{A}_q,\Theta_{p}(\omega)\leftrightarrow_{\sigma}
\partial B(n))$ (note that these two conditional laws are different).  With faces $\Theta_{p}(\omega)$ and
$\Theta_q(\omega')$ fixed, one can determine (\ref{e42}) in the
configuration $\omega'$. For convenience, we define
$$P'_{q,\Theta_{p}(\omega)}(\cdot)=P'_{q,n,\Theta_{p}(\omega)}(\cdot):=P'(\cdot|\mathcal {A}_q,\Theta_{p}(\omega)\leftrightarrow_{\sigma}\partial B(n)),$$
$$E'_{q,\Theta_{p}(\omega)}(\cdot)=E'_{q,n,\Theta_{p}(\omega)}(\cdot):=E'(\cdot|\mathcal {A}_q,\Theta_{p}(\omega)\leftrightarrow_{\sigma}\partial B(n)).$$

It will also be necessary to introduce a copy
$(\Omega'',\mathscr{B}'',P''_{q,\Theta_{p}(\omega)})$ of
$(\Omega',\mathscr{B}',P'_{q,\Theta_{p}(\omega)})$.  The element of
$\Omega''$ is denoted by $\omega''$ and expectation with respect to
$P''_{q,\Theta_{p}(\omega)}$ is denoted by
$E''_{q,\Theta_{p}(\omega)}$.  The following lemma is the analog of
(2.11) in \cite{r1}.  However, we can not get the analogous results
corresponding to (2.12) and Lemma 2 in \cite{r1}: in [17], exact
independence can be revealed, which is not possible here due to the
``global'' conditioning.
\begin{lemma}\label{l4}
We have
\begin{eqnarray*}
&&\Delta_p(\omega)=\theta(\Theta_{p-1}(\omega),\Theta_p(\omega))(\omega)+E_{q,\Theta_{p}(\omega)}''\theta(\Theta_{p}(\omega),\Theta_q(\omega''))(\omega'')\\
&&~~~~~~~~~~~-E_{q,\Theta_{p-1}(\omega)}'\theta(\Theta_{p-1}(\omega),\Theta_q(\omega'))(\omega'),~~1\leq
p\leq q.
\end{eqnarray*}
\end{lemma}

\begin{proof}
The proof is similar to the proof of (2.11) in \cite{r1}.  Fix a
configuration $\omega$. By  (\ref{e21}) and the definition of
$\Delta_p$ in (\ref{e22}), we have
\begin{equation}\label{e23}
\Delta_p=\theta(\Theta_{p-1},\Theta_p)+E_q(\theta(\Theta_{p},\Theta_q)|\mathcal
{F}_p)-E_q(\theta(\Theta_{p-1},\Theta_q)|\mathcal {F}_{p-1}),~~1\leq
p\leq q.
\end{equation}
Combining (\ref{e23}) and
\begin{equation*}
E_q(\theta(\Theta_{p},\Theta_q)|\mathcal
{F}_p)(\omega)=E_{q,\Theta_{p}(\omega)}''\theta(\Theta_{p}(\omega),\Theta_{q}(\omega''))(\omega''),
\end{equation*}
\begin{equation*}
E_q(\theta(\Theta_{p-1},\Theta_q)|\mathcal
{F}_{p-1})(\omega)=E_{q,\Theta_{p-1}(\omega)}'\theta(\Theta_{p-1}(\omega),\Theta_{q}(\omega'))(\omega'),
\end{equation*}
the conclusion follows immediately.
\end{proof}

\begin{lemma}\label{l5}
There exist constants $c_i>0,1\leq i\leq 9$ such that for $q\geq 1$
and $n\geq 2^{q+q^{\frac{1}{3}}+2}$
\begin{equation}\label{e12}
P_q(m(p)-p\geq x)\leq c_1\exp(-c_2x),~~x\geq 0,0\leq p\leq q,
\end{equation}
\begin{equation}\label{e1}
P_q(|\Delta_p|\geq x)\leq c_3\exp(-c_4x),~~x\geq 0,1\leq p\leq q,
\end{equation}
\begin{equation}\label{e2}
P_q\left(\max_{1\leq p\leq q}|\Delta_p|\geq \varepsilon
q^{\frac{1}{2}}\right)\leq c_5q\exp(-c_6\varepsilon
q^{\frac{1}{2}}),~~\varepsilon>0,
\end{equation}
\begin{equation}\label{e3}
E_q\left(\max_{1\leq p\leq q}\Delta_p^2\right)\leq c_7q,
\end{equation}
\begin{equation}\label{e4}
c_8q\leq\sum_{p=1}^{q}E_q\Delta_p^2\leq c_9 q.
\end{equation}
\end{lemma}

\begin{proof}
By the definitions of $\mathcal {A}_q$ and $P_q(\cdot)$, it's
obvious that for $x>q+q^{\frac{1}{3}}-p$, $P_q(m(p)-p\geq x)=0$.  By
(\ref{e19}) and (\ref{e18}), there exist constants $c_1,c_2>0$ such
that for $0\leq x\leq q+q^{\frac{1}{3}}-p$,
\begin{eqnarray*}
&&P_q(m(p)-p\geq x)\leq c_1P(m(p)-p\geq x|\partial
B(1)\leftrightarrow_{\sigma}\partial B(n))\leq c_1\exp(-c_2x).
\end{eqnarray*}
Then we conclude (\ref{e12}).

Using the symmetric property of the polychromatic setting of the
windings and coupling argument, one may give a short proof of
(\ref{e1}). However, the following method we used to prove
(\ref{e1}) can be modified easily to prove the analogous result of
the 1-arm case.

By Lemma \ref{l4}, for $1\leq p\leq q$,
\begin{eqnarray}
&&|\Delta_p(\omega)|\leq|E'_{q,\Theta_{p-1}(\omega)}\theta(\Theta_{p-1}(\omega),\Theta_q(\omega'))(\omega')-E''_{q,\Theta_{p}(\omega)}\theta(\Theta_{p}(\omega),\Theta_q(\omega''))(\omega'')|\nonumber\\
&&~~~~~~~~~~~~~+|\theta(\Theta_{p-1}(\omega),\Theta_p(\omega))(\omega)|.\label{e31}
\end{eqnarray}

For fixed $\omega$, now we will prove there exists a constant
$c_{10}>0$ such that
\begin{equation}\label{e24}
|E'_{q,\Theta_{p-1}(\omega)}\theta(\Theta_{p-1}(\omega),\Theta_q(\omega'))-E''_{q,\Theta_{p}(\omega)}\theta(\Theta_{p}(\omega),\Theta_q(\omega''))|\leq
c_{10}(m(p,\omega)-p+1).
\end{equation}
\begin{remark}\label{note}
When $p=1$, we just let
$P'_{q,\Theta_{p-1}(\omega)}=P'_q,E'_{q,\Theta_{p-1}(\omega)}=E'_q$
, the arguments in the following also adapt to this case.
\end{remark}

First we show there exist $c_{11},c_{12},c_{13}>0$ such that for
$0\leq p_1<p_2\leq q$ and $x\geq c_{13}(p_2-p_1)$,
\begin{equation}\label{e26}
P_{q,\Theta_{p_1}(\omega)}'(|\theta(\Theta_{p_1}(\omega),\Theta_{p_2}(\omega'))|\geq
x)\leq c_{11}\exp(-c_{12}x).
\end{equation}
Let us note first that,
\begin{eqnarray*}
&&P_{q,\Theta_{p_1}(\omega)}'(|\theta(\Theta_{p_1}(\omega),\Theta_{p_2}(\omega'))|\geq
x)\\
&&\leq P_{q,\Theta_{p_1}(\omega)}'(m(p_2,\omega')-p_2\geq
c_{14}x)\\
&&~~~~+P_{q,\Theta_{p_1}(\omega)}'(m(p_2,\omega')-p_2\leq c_{14}x
,|\theta(\Theta_{p_1}(\omega),\Theta_{p_2}(\omega'))|\geq x),
\end{eqnarray*}
where $c_{14}$ will be fixed later.  Let us now estimate each term
separately.  For the first term, by Lemma \ref{l9}, similar to the
proof of (\ref{e12}),  we get that there exist $c_{15},c_{16}>0$
such that for $x\geq 0$,
\begin{equation*}
P_{q,\Theta_{p_1}(\omega)}'(m(p_2,\omega')-p_2\geq x)\leq
c_{15}\exp(-c_{16}x).
\end{equation*}
For short, let $p_2':=\min\{p_2+\lfloor c_{14}x \rfloor+1,q+\lfloor
q^{\frac{1}{3}} \rfloor\}$.  Recall
$R(m,n):=\{z\in\mathbb{Z}^2:|\arg(z)|<\frac{\pi}{10}\}\cap A(m,n)$
and we call a path in $R(m,n)$ crosses $R(m,n)$ if it connects the
two rays of argument $\pm\frac{\pi}{10}$. For the second term,
\begin{eqnarray*}
&&P_{q,\Theta_{p_1}(\omega)}'(m(p_2,\omega')-p_2\leq  c_{14}x
,|\theta(\Theta_{p_1}(\omega),\Theta_{p_2}(\omega'))|\geq x)\\
&&\leq\frac{c_{17}P'(m(p_2,\omega')-p_2\leq  c_{14}x
,|\theta(\Theta_{p_1}(\omega),\Theta_{p_2}(\omega'))|\geq
x,\Theta_{p_1}(\omega)\leftrightarrow_{\sigma}\partial
B(n))}{P'(\Theta_{p_1}(\omega)\leftrightarrow_{\sigma}
\partial B(n))}~~\mbox{by (\ref{e19})}\\
&&\leq\frac{c_{18}P'(m(p_2,\omega')-p_2\leq  c_{14}x
,|\theta(\Theta_{p_1}(\omega),\Theta_{p_2}(\omega'))|\geq
x)}{P'(\Theta_{p_1}(\omega)\leftrightarrow_{\sigma}
\partial B(2^{p_2'}))}~~\mbox{by quasi-multiplicativity}\\
&&\leq\frac{c_{50}P'(\mbox{there exist at least $\frac{x}{2\pi}-4$
disjoint black crossing paths in
$R(2^{p_1-1},2^{p_2'})$})}{P'(\Theta_{p_1}(\omega)\leftrightarrow_{\sigma}
\partial B(2^{p_2'}))}.
\end{eqnarray*}
Since $x\geq c_{13}(p_2-p_1)$, by Lemma \ref{l1} and (\ref{e41}),
then we can choose $c_{13},c_{14}$ appropriately, such that
\begin{equation*}
P_{q,\Theta_{p_1}(\omega)}'(m(p_2,\omega')-p_2\leq  c_{14}x
,|\theta(\Theta_{p_1}(\omega),\Theta_{p_2}(\omega'))|\geq x)\leq
c_{19}\exp(-c_{20}x).
\end{equation*}
Now we have bounded the two terms and completed the proof of
(\ref{e26}).

To obtain inequality (\ref{e24}), let us consider the two cases in
the following.

Case 1 ($m(p,\omega)\geq q-1$).  By (\ref{e26}), for all $x\geq
c_{13}(q-p+1)$,
\begin{equation*}
P_{q,\Theta_{p-1}(\omega)}'(|\theta(\Theta_{p-1}(\omega),\Theta_q(\omega'))|\geq
x)\leq c_{11}\exp(-c_{12}x).
\end{equation*}
Thus there exists $c_{21}>0$ such that
$$E'_{q,\Theta_{p-1}(\omega)}|\theta(\Theta_{p-1}(\omega),\Theta_q(\omega'))|\leq c_{21}(q-p+1).$$
Similarly we have
$$E''_{q,\Theta_{p}(\omega)}|\theta(\Theta_{p}(\omega),\Theta_q(\omega''))|\leq c_{22}(q-p+1).$$
Combining the above two inequalities, choosing $c_{10}$
appropriately, we obtain (\ref{e24}) since $m(p,\omega)\geq q-1$.

Case 2 ($m(p,\omega)\leq q-2$).  By the coupling result of Lemma
\ref{l2} and (\ref{e19}), there exists a constant $c_{23}>0$ such
that for $1\leq p\leq q-1$, we can couple
$P_{q,\Theta_{p-1}(\omega)}'$ and $P_{q,\Theta_{p}(\omega)}''$ so
that with probability at least $1-\exp(-c_{23} x)$, there exists
$l(p,\omega,\omega')=l(p,\omega,\omega'')$, such that
$m(p,\omega)+1\leq l(p,\omega,\omega')\leq m(p,\omega)+1+x\leq q$,
the good faces $\Theta_{l(p,\omega,\omega')}(\omega')$ and
$\Theta_{l(p,\omega,\omega'')}(\omega'')$ are identical, and the
configurations in $B(n)\backslash
\overline{\Theta}_{l(p,\omega,\omega')}(\omega')$ and
$B(n)\backslash \overline{\Theta}_{l(p,\omega,\omega'')}(\omega'')$
are also identical.  Denote by $\mathcal {S}$ the event that the
above coupling succeeds for $m(p,\omega)+1\leq
l(p,\omega,\omega')\leq q$.

Let us note first that,
\begin{eqnarray*}
&&|E'_{q,\Theta_{p-1}(\omega)}\theta(\Theta_{p-1}(\omega),\Theta_q(\omega'))-E''_{q,\Theta_{p}(\omega)}\theta(\Theta_{p}(\omega),\Theta_q(\omega''))|\\
&&\leq E'_{q,\Theta_{p-1}(\omega)}|I_{\mathcal
{S}}\theta(\Theta_{p-1}(\omega),\Theta_{l(p,\omega,\omega')}(\omega'))|+E''_{q,\Theta_{p}(\omega)}|I_{\mathcal
{S}}\theta(\Theta_{p}(\omega),\Theta_{l(p,\omega,\omega')}(\omega''))|\\
&&~~~~+E'_{q,\Theta_{p-1}(\omega)}|I_{\mathcal
{S}^c}\theta(\Theta_{p-1}(\omega),\Theta_q(\omega'))|+E''_{q,\Theta_{p}(\omega)}|I_{\mathcal
{S}^c}\theta(\Theta_p(\omega),\Theta_q(\omega''))|.
\end{eqnarray*}
We now estimate the four terms separately.

For short, define $$M(p,\omega):=min\{m(p,\omega)+2+\lfloor
c_{24}x\rfloor, q+1\},$$
$$\mathcal {B}:=\{\mbox{there exist at
least $\frac{x}{2\pi}-4$ disjoint black crossing paths in
$R(2^{p-1},2^{M(p,\omega)})$}\},$$ where $c_{24}$ will be fixed
later. For $x\geq c_{25}(m(p,\omega)-p+1)$, we have the following
inequality, where $c_{25}$ will be fixed later.
\begin{eqnarray*}
&&P_{q,\Theta_{p-1}(\omega)}'(\mathcal
{S},|\theta(\Theta_{p-1}(\omega),\Theta_{l(p,\omega,\omega')}(\omega'))|\geq
x)\\
&&\leq P_{q,\Theta_{p-1}(\omega)}'(\mathcal
{S},l(p,\omega,\omega')\geq
m(p,\omega)+1+c_{24}x)+P_{q,\Theta_{p-1}(\omega)}'(\mathcal {B})\\
&&\leq c_{40}\exp(-c_{26}x)+\frac{c_{27}P'(\mathcal
{B},\Theta_p(\omega)\leftrightarrow_{\sigma}\partial
B(n))}{P'(\Theta_p(\omega)\leftrightarrow_{\sigma}
\partial B(n))}~~\mbox{by (\ref{e19}) and coupling result}
\\
&&\leq c_{40}\exp(-c_{26}x)+\frac{c_{28}P'(\mathcal
{B})}{P'(\Theta_p(\omega)\leftrightarrow_{\sigma}
\partial B(2^{M(p,\omega)}))}~~\mbox{by quasi-multiplicativity.}
\end{eqnarray*}
By Lemma \ref{l1} and (\ref{e41}), then we can choose appropriate
$c_{24},c_{25}$ such that
$$P_{q,\Theta_{p-1}(\omega)}'(\mathcal
{S},|\theta(\Theta_{p-1}(\omega),\Theta_{l(p,\omega,\omega')}(\omega'))|\geq
x)\leq c_{29}\exp(-c_{30}x).$$  Choosing $c_{10}$ large enough, then
we obtain
\begin{equation*}
E'_{q,\Theta_{p-1}(\omega)}|I_{\mathcal
{S}}\theta(\Theta_{p-1}(\omega),\Theta_{l(p,\omega,\omega')}(\omega'))|\leq
c_{10}(m(p,\omega)-p+1)/4.
\end{equation*}
Similarly
\begin{equation*}
E''_{q,\Theta_{p}(\omega)}|I_{\mathcal
{S}}\theta(\Theta_{p}(\omega),\Theta_{l(p,\omega,\omega')}(\omega''))|\leq
c_{10}(m(p,\omega)-p+1)/4.
\end{equation*}

For the third term, let $c_{10}$ be a large enough constant, then
\begin{eqnarray*}
&&E'_{q,\Theta_{p-1}(\omega)}|I_{\mathcal
{S}^c}\theta(\Theta_{p-1}(\omega),\Theta_q(\omega'))|\\
&&\leq [P_{q,\Theta_{p-1}(\omega)}'(\mathcal
{S}^c)]^{\frac{1}{2}}[E'_{q,\Theta_{p-1}(\omega)}|\theta(\Theta_{p-1}(\omega),\Theta_q(\omega'))|^2]^{\frac{1}{2}}~~\mbox{by Cauchy-Schwarz inequality}\\
&&\leq c_{31}\exp(-c_{32}(q-1-m(p,\omega)))(q-p+1)~~\mbox{by the coupling result and (\ref{e26})}\\
&&\leq c_{10}(m(p,\omega)-p+1)/4~~\mbox{by $m(p,\omega)\leq q-2$}.
\end{eqnarray*}
Similarly we have
\begin{equation*}
E''_{q,\Theta_{p}(\omega)}|I_{\mathcal
{S}^c}\theta(\Theta_p(\omega),\Theta_q(\omega''))|\leq
c_{10}(m(p,\omega)-p+1)/4.
\end{equation*}
Now the four terms have been bounded, which ends the proof of
(\ref{e24}).

By (\ref{e31}) and (\ref{e24}), we can choose an appropriate
constant $c_{33}$ such that
\begin{equation*}
P_q(|\Delta_p|\geq x)\leq P_q(m(p,\omega)-p\geq
c_{33}x)+P_q(|\theta(\Theta_{p-1}(\omega),\Theta_p(\omega))|\geq
x/2).
\end{equation*}
Now we bound the two terms in the r.h.s of above inequality.  For
the first term, by (\ref{e12}) we get
\begin{equation*}
P_q(m(p,\omega)-p\geq c_{33}x)\leq c_1\exp(-c_2c_{33}x).
\end{equation*}
For the second term, if $2\leq p\leq q$, by (\ref{e26}), there exist
$c_{34},c_{35}>0$ such that
\begin{eqnarray*}
&&P_q(|\theta(\Theta_{p-1}(\omega),\Theta_p(\omega))(\omega)|\geq
x/2)\\
&&=\sum_{\Theta_{p-1}(\omega)}P_q(\Theta_{p-1}(\omega))P_{q,\Theta_{p-1}(\omega)}'(\theta(\Theta_{p-1}(\omega),\Theta_{p}(\omega'))>x/2)\leq
c_{34}\exp(-c_{35}x);
\end{eqnarray*}
if $p=1$, we can bound the second term directly by (\ref{e26}) and
Note \ref{note}.  Thus (\ref{e1}) is concluded.
 Using (\ref{e1}), we conclude
(\ref{e2}),(\ref{e3}) and the second inequality in (\ref{e4})
immediately.  Now let us prove the first inequality in (\ref{e4}).
By Lemma \ref{l4} and (\ref{e24}),  we have
$$\Delta_p\geq \theta(\Theta_{p-1},\Theta_p)-c_{10}(m(p)-p+1).$$
Applying (\ref{e19}) and inequality in the above, gives
\begin{eqnarray}
&&E_q\Delta_p^2\geq P_q(\Delta_p\geq 1)\nonumber\\
&&\geq P_q(m(p-1)=p-1,m(p)=p+1,
\theta(\Theta_{p-1},\Theta_{p})\geq 2c_{10}+1)\nonumber\\
&&\geq P(m(p-1)=p-1,m(p)=p+1, \theta(\Theta_{p-1},\Theta_{p})\geq
2c_{10}+1)|\partial B(1)\leftrightarrow_{\sigma}\partial
B(n))\nonumber\\
&&~~~~-c_{41}\exp(-c_{36}q^{\frac{1}{3}})\label{e13}
\end{eqnarray}
Define event
\begin{eqnarray*}
&&\mathcal {A}:=\{\mbox{$A(p-1)$ and $A(p+1)$ have good faces, $\Theta_{p-1}\leftrightarrow_{\sigma}\Theta_{p+1}$, }\\
&&~~~~~~~~\mbox{$\theta(\Theta_{p-1},\Theta_{p})\geq 2c_{10}+1$,
there exist five disjoint paths crossing $A(p)$}\}.
\end{eqnarray*}
Assume $A(p-1)$ and $A(p+1)$ have good faces.  By Lemma \ref{l3},
the four interfaces crossing $A(p+1)$ also induce ``exterior good
faces" $\Theta_{p+1}'$ around $\partial B(2^{p+1})$ in $A(p+1)$, see
Fig. \ref{fig1}. By RSW, FKG and Lemma \ref{l3}, we can connect
$\Theta_{p+1}'$ and $\Theta_{p-1}$ by four alternating color arms
with large enough winding angles and obtain
\begin{equation}\label{e14}
P(\mathcal {A})\geq c_{37}.
\end{equation}
Then by FKG, RSW and Separation Lemma (see Theorem 11 in
\cite{r17}), using standard gluing argument, we have
\begin{eqnarray}
&&P(\mathcal {A},\partial B(1)\leftrightarrow_{\sigma}\partial
B(n))\nonumber\\
&&~~~~\geq c_{38}P(\mathcal {A})P(\partial
B(1)\leftrightarrow_{\sigma}\partial B(2^{p-2}))P(\partial
B(2^{p+3})\leftrightarrow_{\sigma}\partial B(n)).\label{e15}
\end{eqnarray}
Combining (\ref{e13}), (\ref{e14}), (\ref{e15}), we know that there
exists a constant $c_{39}>0$, such that for all large $q$, we have
$$E_q\Delta_p^2\geq c_{39},~~4\leq p\leq q-4.$$
Thus the first inequality in (\ref{e4}) follows.
\end{proof}
\begin{lemma}\label{l8}
Denote by $E_n(\cdot)$ the expectation with respect to
$P(\cdot|\partial B(1)\leftrightarrow_{\sigma}\partial B(n))$.  For
all large $q$ and all $n\geq 2^{q+q^{\frac{1}{3}}+2}$,
\begin{equation*}
|E_{q,n}\theta(\partial B(1),\Theta_{q,n})|\leq 2\pi\mbox{ and
}|E_n\theta_n|\leq 2\pi.
\end{equation*}
\end{lemma}
\begin{proof}
Define $$\theta_{max}:=\max\{\theta(\gamma):\gamma\mbox{ is an arm
connecting $\partial B(1)$ and $\Theta_{q,n}$}\},$$
$$\theta_{min}:=\min\{\theta(\gamma):\gamma\mbox{ is an arm
connecting $\partial B(1)$ and $\Theta_{q,n}$}\}.$$  By (\ref{e26}),
it is easy to see that $E_{q,n}\theta(\partial B(1),\Theta_{q,n})$
exists.  Since the winding angles of the arms between $\partial
B(1)$ and $\Theta_{q,n}$ differ at most $2\pi$, hence
$E_{q,n}\theta_{max},E_{q,n}\theta_{min}$ also exist.  By the
symmetry of the lattice and the definition of $\Theta_{q,n}$, it is
obvious that $E_{q,n}\theta_{max}=-E_{q,n}\theta_{min}.$  Hence
$E_{q,n}[\theta_{max}+\theta_{min}]=0$. Then we conclude
$|E_{q,n}\theta(\partial B(1),\Theta_{q,n})|\leq 2\pi$ by
$$|\theta(\partial B(1),\Theta_{q,n})-(\theta_{max}+\theta_{min})/2|\leq 2\pi.$$
$|E_n\theta_n|\leq 2\pi$ can be proved similarly.
\end{proof}
\begin{lemma}\label{l7}
Assume $n\geq 2^{q+q^{\frac{1}{3}}+2}$. Under $P_{q,n}$, as
$q\rightarrow \infty$,
\begin{equation*}
\frac{\theta(\partial
B(1),\Theta_{q,n})}{\left(\sum_{p=1}^{q}E_{q,n}\Delta_{p,q,n}^2\right)^{1/2}}\rightarrow_d
N(0,1).
\end{equation*}
\end{lemma}
\begin{proof}
By Lemma \ref{l8} and (\ref{e4}), Lemma \ref{l7} is equivalent to
\begin{equation}\label{e17}
\frac{\theta(\partial B(1),\Theta_{q,n})-E_{q,n}\theta(\partial
B(1),\Theta_{q,n})}{\left(\sum_{p=1}^{q}E_{q,n}\Delta_{p,q,n}^2\right)^{1/2}}\rightarrow_d
N(0,1).
\end{equation}
Let us now check the three conditions of Theorem 2.3 in \cite{r7}.
First we set
$$X_{p,q,n}:=\frac{\Delta_{p,q,n}}{\left(\sum_{p=1}^{q}E_{q,n}\Delta_{p,q,n}^2\right)^{1/2}},$$
then we can write
$$\frac{\theta(\partial
B(1),\Theta_{q,n})-E_{q,n}\theta(\partial
B(1),\Theta_{q,n})}{\left(\sum_{p=1}^{q}E_{q,n}\Delta_{p,q,n}^2\right)^{1/2}}=\sum_{p=1}^{q}X_{p,q,n}.$$
By (\ref{e4}),
$$|X_{p,q,n}|\leq |\Delta_{p,q,n}|/(c_8q)^{1/2}.$$
The conditions (2.3a) and (2.3b) of McLeish \cite{r7} are implied by
(\ref{e2}) and (\ref{e3}). By (\ref{e4}), the condition (2.3c) is
equivalent to
\begin{equation}\label{e16}
\frac{1}{q}\sum_{p=1}^{q}(\Delta_{p,q,n}^2-E_{q,n}\Delta_{p,q,n}^2)\rightarrow
0~~\mbox{in probability.}
\end{equation}
We can not use the method from the proof of (2.60) in \cite{r1}
since there is a lack of independence of our setting.  Thanks to the
coupling result, we can gain some weak independence of our model.
Recall that we let $E_q=E_{q,n}$ and $\Delta_p=\Delta_{p,q,n}$ for
short.  Note that if we prove
\begin{equation}\label{e27}
E_q\left[\sum_{p=1}^{q}(\Delta_p^2-E_q\Delta_p^2)\right]^2=o(q^2),
\end{equation}
then (\ref{e16}) follows.  For a fixed constant $c>0$, let us split
the above sum into two terms:
\begin{eqnarray*}
&&E_q\left[\sum_{p=1}^{q}(\Delta_p^2-E_q\Delta_p^2)\right]^2=\sum_{|p-r|\leq
c\log
q}E_q[(\Delta_p^2-E_q\Delta_p^2)(\Delta_r^2-E_q\Delta_r^2)]\\
&&~~~~~~~~~~~~~~~~~~~~~~~~~~~~~~~~~+\sum_{|p-r|> c\log
q}E_q[(\Delta_p^2-E_q\Delta_p^2)(\Delta_r^2-E_q\Delta_r^2)].
\end{eqnarray*}
For the first term, using (\ref{e1}), we have
$$\sum_{|p-r|\leq
c\log
q}E_q[(\Delta_p^2-E_q\Delta_p^2)(\Delta_r^2-E_q\Delta_r^2)]\leq
c_1q\log q.$$ Now we estimate the second term. Assume $r-p> c\log q$
in the following. Define events
\begin{equation*}
\mathcal {A}:=\{m(p,\omega)\leq p+\frac{c}{2}\log q\},
\end{equation*}
\begin{eqnarray*}
&&\mathcal {B}:=\{\mbox{there exists $m(p,\omega)+1\leq
l=l(p,\omega,\omega')\leq r-2$, we can couple
$P_{q,\Theta_p(\omega)}'$}\\
&&~~~~~~~~~\mbox{and $P_q''$ such that $\Theta_l(\omega')$ and
$\Theta_l(\omega'')$ are identical, and the configurations}\\
&&~~~~~~~~~\mbox{in $B(n)\backslash \overline{\Theta}_l(\omega')$
and $B(n)\backslash \overline{\Theta}_l(\omega'')$ are also
identical}\}.
\end{eqnarray*}
First, we write
\begin{eqnarray}
&&E_q[(\Delta_p^2-E_q\Delta_p^2)(\Delta_r^2-E_q\Delta_r^2)]\nonumber\\
&&=E_q[I_{\mathcal
{A}^c}(\Delta_p^2-E_q\Delta_p^2)(\Delta_r^2-E_q\Delta_r^2)]+E_q[I_{\mathcal
{A}}(\Delta_p^2-E_q\Delta_p^2)(\Delta_r^2-E_q\Delta_r^2)].\label{e32}
\end{eqnarray}
Let us estimate the two terms in the r.h.s. of above inequality
separately. For the first term, with H\"{o}lder's inequality,
(\ref{e12}) and (\ref{e1}), we get
\begin{eqnarray}
&&E_q[I_{\mathcal
{A}^c}(\Delta_p^2-E_q\Delta_p^2)(\Delta_r^2-E_q\Delta_r^2)]\leq
P(\mathcal
{A}^c)^{\frac{1}{3}}(E_q|\Delta_p^2-E_q\Delta_p^2|^3)^{\frac{1}{3}}(E_q|\Delta_r^2-E_q\Delta_r^2|^3)^{\frac{1}{3}}\nonumber\\
&&~~~~~~~~~~~~~~~~~~~~~~~~~~~~~~~~~~~~~~~~~~~~~~~~~~~~\leq
c_2\exp(-c_3\log q)=o(1).\label{e33}
\end{eqnarray}
For the second term, write
\begin{eqnarray}
&&E_q[I_{\mathcal
{A}}(\Delta_p^2-E_q\Delta_p^2)(\Delta_r^2-E_q\Delta_r^2)]\nonumber\\
&&=E_q[I_{\mathcal
{A}}(\Delta_p^2(\omega)-E_q\Delta_p^2(\omega))E_{q,\Theta_p(\omega)}'[I_{\mathcal
{B}}(\Delta_r^2(\omega')-E_q'\Delta_r^2(\omega'))]]\nonumber\\
&&~~~+E_q[I_{\mathcal
{A}}(\Delta_p^2(\omega)-E_q\Delta_p^2(\omega))E_{q,\Theta_p(\omega)}'[I_{\mathcal
{B}^c}(\Delta_r^2(\omega')-E_q'\Delta_r^2(\omega'))]]\label{e46}.
\end{eqnarray}
By the Cauchy-Schwarz inequality, we have
\begin{eqnarray}
&&E_{q,\Theta_p(\omega)}'|I_{\mathcal
{B}^c}(\Delta_r^2(\omega')-E_q'\Delta_r^2(\omega'))|\nonumber\\
&&~~~~~~~~~~~\leq [P_{q,\Theta_p(\omega)}'(\mathcal
{B}^c)]^{\frac{1}{2}}[E_{q,\Theta_p(\omega)}'[\Delta_r^2(\omega')-E_q'\Delta_r^2(\omega')]^2]^{\frac{1}{2}}.\label{e45}
\end{eqnarray}
Using a very similar argument as the proof of (\ref{e1}), we know
that there exist universal constants $c_4,c_5>0$ such that for all
$x\geq 0$,
$$P_{q,\Theta_p(\omega)}'(|\Delta_r(\omega')|\geq x)\leq c_4\exp(-c_5x).$$
Combining (\ref{e45}), the coupling result, (\ref{e19}) and above
inequality, we get
\begin{equation}\label{e47}
E_{q,\Theta_p(\omega)}'[I_{\mathcal
{B}^c}(\Delta_r^2(\omega')-E_q'\Delta_r^2(\omega'))]=o(1).
\end{equation}
Similarly we have $$E_q''[I_{\mathcal
{B}^c}(\Delta_r^2(\omega'')-E_q''\Delta_r^2(\omega''))]=o(1).$$ Then
$E_q''[\Delta_r^2(\omega'')-E_q''\Delta_r^2(\omega'')]=0$ implies
$$E_q''[I_{\mathcal
{B}}(\Delta_r^2(\omega'')-E_q''\Delta_r^2(\omega''))]=o(1).$$ By the
definition of $\mathcal {B}$, we immediately obtain
\begin{equation}\label{e48}
E_{q,\Theta_p(\omega)}'[I_{\mathcal
{B}}(\Delta_r^2(\omega')-E_q'\Delta_r^2(\omega'))]=E_q''[I_{\mathcal
{B}}(\Delta_r^2(\omega'')-E_q''\Delta_r^2(\omega''))]=o(1).
\end{equation}
Combining (\ref{e46}), (\ref{e47}), (\ref{e48}) and (\ref{e1}),  we
get
\begin{equation}\label{e34}
E_q[I_{\mathcal
{A}}(\Delta_p^2-E_q\Delta_p^2)(\Delta_r^2-E_q\Delta_r^2)]=E_q[I_{\mathcal
{A}}(\Delta_p^2-E_q\Delta_p^2)]o(1)=o(1).
\end{equation}
Consequently, for the second term, by (\ref{e32}), (\ref{e33}) and
(\ref{e34}), we get
$$\sum_{|p-r|> c\log
q}E_q[(\Delta_p^2-E_q\Delta_p^2)(\Delta_r^2-E_q\Delta_r^2)]=o(q^2).$$
\end{proof}

\begin{proof}[Proof of Theorem \ref{t1}]
Let $2^{q+q^{\frac{1}{3}}+2}\leq n<2^{(q+1)+(q+1)^{\frac{1}{3}}+2}$
and $0<\varepsilon<\frac{1}{6}$. Define event $$\mathcal
{B}:=\{\mbox{there exist at least
$q^{\frac{1}{2}-\varepsilon}/(2\pi)-4$ disjoint black crossing paths
in $R(2^q,n)$}\}.$$  Recall we let $P_q=P_{q,n}$ and
$\Delta_p=\Delta_{p,q,n}$ for short. By the definition of $\theta_n$
and $\theta(\cdot,\cdot)$, we obtain
\begin{eqnarray}
&&P_q(|\theta_n-\theta(\partial B(1),\Theta_q)|\geq
q^{\frac{1}{2}-\varepsilon})\nonumber\\
&&\leq P_q(\mathcal {B})\nonumber\\
&&\leq \frac{c_1P(\mathcal {B},\partial B(1)\leftrightarrow_{\sigma}\partial B(n))}{P(\partial B(1)\leftrightarrow_{\sigma}\partial B(n))}~~\mbox{by (\ref{e19})}\nonumber\\
&&\leq\frac{c_2P(\mathcal {B},\partial
B(2^q)\leftrightarrow_{\sigma}\partial B(n))}{P(\partial
B(2^q)\leftrightarrow_{\sigma}\partial B(n))}~~\mbox{by
quasi-multiplicativity}\nonumber\\
&&\leq c_3\exp(-c_4q^{\frac{1}{2}-\varepsilon})~~\mbox{by Lemma
\ref{l1} and (\ref{e40})}.\label{e36}
\end{eqnarray}
For $2^{q+q^{\frac{1}{3}}+2}\leq n<2^{q+(q+1)^{\frac{1}{3}}+3}$,
define
\begin{equation}\label{e43}
h_n:=\left(\sum_{p=1}^{q}E_q\Delta_p^2\right)^{1/2}.
\end{equation}
Then by (\ref{e36}), Lemma \ref{l7}, (\ref{e19}) and (\ref{e4}),
under $P(\cdot|\partial B(1)\leftrightarrow_{\sigma}
\partial B(n))$ we have
$$\frac{\theta_n}{h_n}\rightarrow_d
N(0,1).$$ Hence Theorem \ref{t1} is concluded if
$a_n=h_n+o(\sqrt{\log n)}$. Let us prove this now.  For short, let
$\theta_q:=\theta(\partial B(1),\Theta_q)$.  By Lemma \ref{l8} ,
(\ref{e19}) and (\ref{e4}),
\begin{equation}\label{e37}
h_n^2=E_q[\theta_q^2]-[E_q\theta_q]^2=E_q[\theta_q^2]+O(1)=(1+o(1))E[\theta_q^2I_{\mathcal
{A}_q}]\asymp \log n,
\end{equation}
\begin{equation}\label{e39}
a_n^2=E[\theta_n^2]-[E\theta_n]^2=E[\theta_n^2]+O(1)=E[\theta_n^2I_{\mathcal
{A}_q}]+E[\theta_n^2I_{\mathcal {A}_q^c}]+O(1).
\end{equation}
By Lemma \ref{l1} and (\ref{e40}), it is easy to show that there
exist $c_5,c_6,c_7>0$, such that for $x>c_5\log n$,
\begin{equation}\label{e35}
P(|\theta_n|\geq x|\partial B(1)\leftrightarrow_{\sigma} \partial
B(n))\leq c_6\exp(-c_7x).
\end{equation}
 By the Cauchy-Schwarz inequality, (\ref{e35}) and (\ref{e19}),
\begin{equation}\label{e38}
E[\theta_n^2I_{\mathcal {A}_q^c}]\leq
(E[\theta_n^4])^{1/2}(P(\mathcal {A}_q^c))^{1/2}\leq c_8(\log
n)^2\exp(-c_9(\log n)^{1/3})=o(1).
\end{equation}
By (\ref{e37}), (\ref{e39}) and (\ref{e38}), we have
\begin{equation*}
|a_n^2-h_n^2|=|E[\theta_n^2I_{\mathcal
{A}_q}]-E[\theta_q^2I_{\mathcal {A}_q}]|+o(\log n).
\end{equation*}
In the following we prove $|E[\theta_n^2I_{\mathcal
{A}_q}]-E[\theta_q^2I_{\mathcal {A}_q}]|=o(\log n)$, which implies
$a_n=h_n+o(\sqrt{\log n)}$.  Analogous to the proof of (\ref{e36}),
it can be shown that there exist $c_{10},c_{11},c_{12}>0$, such that
for all $x\geq c_{10}(\log n)^{1/3}$,
\begin{equation}\label{e49}
P(|\theta_nI_{\mathcal {A}_q}-\theta_qI_{\mathcal {A}_q}|\geq
x|\partial B(1)\leftrightarrow_{\sigma}
\partial B(n))\leq c_{11}\exp(-c_{12}x).
\end{equation}
Then by the Cauchy-Schwarz inequality, (\ref{e37}), (\ref{e49}) and
(\ref{e4}), we have
\begin{eqnarray*}
&&|E[\theta_n^2I_{\mathcal {A}_q}]-E[\theta_q^2I_{\mathcal
{A}_q}]|\\
&&~~~~=|2E[\theta_qI_{\mathcal {A}_q}(\theta_nI_{\mathcal
{A}_q}-\theta_qI_{\mathcal {A}_q})]+E[\theta_nI_{\mathcal
{A}_q}-\theta_qI_{\mathcal {A}_q}]^2|\\
&&~~~~\leq 2(E[\theta_qI_{\mathcal
{A}_q}]^2)^{1/2}(E[\theta_nI_{\mathcal {A}_q}-\theta_qI_{\mathcal
{A}_q}]^2)^{1/2}+E[\theta_nI_{\mathcal {A}_q}-\theta_qI_{\mathcal
{A}_q}]^2\\
&&~~~~\leq c_{13}(\log n)^{5/6}+c_{14}(\log n)^{2/3}.
\end{eqnarray*}
\end{proof}
\begin{proof}[Proof of Corollary \ref{c1}]
Let $10\leq n\leq m$ and $2^{q+q^{\frac{1}{3}}+1}\leq n\leq
2^{q+q^{\frac{1}{3}}+2}$.  By a slight modification of Proposition
3.6 in \cite{r4} and its proof (analog to the coupling result in
Lemma \ref{l2}), there exists a universal constant $c_1>0$, we can
couple $P(\cdot|\partial B(1)\leftrightarrow
\partial B(n))$ and $P(\cdot|\partial B(1)\leftrightarrow \partial
B(m))$ such that with probability at least
$1-\exp(-c_1q^{\frac{1}{3}})$, there exist identical exterior faces
$\Theta$ with quality $Q(\Theta)\geq \frac{1}{4}$ (well separated)
around $\partial B(2^{q+p})$ for some $0\leq p\leq q^{\frac{1}{3}}$
and identical configurations on $\overline{\Theta}$ for
$P(\cdot|\partial B(1)\leftrightarrow
\partial B(n))$ and $P(\cdot|\partial B(1)\leftrightarrow
\partial B(m))$. Let $\mathcal {A}$ denote the event that the above coupling succeeds and $P_{n,m}(\cdot)$ denote the coupling measure.
For a configuration of $P(\cdot|\partial B(1)\leftrightarrow
\partial B(m))$, we denote by $\theta_{n,m}$ the winding angle of the arm (chosen
uniquely by some definite way) connecting $\partial B(1)$ and
$\partial B(n)$.  Let $0<\varepsilon<\frac{1}{6}$. Define event
$$\mathcal {B}:=\{\mbox{there exist at least
$q^{\frac{1}{2}-\varepsilon}/(4\pi)-4$ disjoint black crossing paths
in $R(2^q,n)$}\}.$$ Then by the coupling argument we discuss above,
\begin{eqnarray*}
&&P_{n,m}(|\theta_{n,m}-\theta_{n,n}|>2q^{\frac{1}{2}-\varepsilon})\\
&&~~~~\leq P_{n,m}(\mathcal {A}^c)+P_{n,m}(\mathcal {A},|\theta_{n,m}-\theta_{n,n}|>2q^{\frac{1}{2}-\varepsilon})\\
&&~~~~\leq \exp(-c_1q^{\frac{1}{3}})+P(\mathcal {B}|\partial
B(1)\leftrightarrow_{\sigma}
\partial B(n))+P(\mathcal {B}|\partial B(1)\leftrightarrow_{\sigma}
\partial B(m)).
\end{eqnarray*}
Since
\begin{eqnarray*}
&&P(\mathcal {B}|\partial B(1)\leftrightarrow_{\sigma}
\partial B(m))=\frac{P(\mathcal {B},\partial B(1)\leftrightarrow_{\sigma}
\partial B(m))}{P(\partial B(1)\leftrightarrow_{\sigma}
\partial B(m))}\\
&&~~~~~~~~~~~~~~~~~~~~~~~~~~~~~~~~~~~~\leq \frac{c_2P(\mathcal
{B})}{P(\partial B(2^q)\leftrightarrow_{\sigma}
\partial B(n))}~~\mbox{by quasi-multiplicativity}\\
&&~~~~~~~~~~~~~~~~~~~~~~~~~~~~~~~~~~~~\leq
c_3\exp(-c_4q^{\frac{1}{2}-\varepsilon})~~\mbox{by Lemma \ref{l1}
and (\ref{e40})},
\end{eqnarray*}
then we get
$$P_{n,m}(|\theta_{n,m}-\theta_{n,n}|>q^{\frac{1}{2}-\varepsilon})=o(1)~~\mbox{as}~~q\rightarrow\infty.$$
Recall $a_n\asymp \sqrt{\log n}$, $q\asymp \log n$.  Using Theorem
\ref{t1}, under $P(\cdot|\partial B(1)\leftrightarrow
\partial B(m))$, let $m\geq n\rightarrow\infty$, we have
\begin{equation*}
\frac{\theta_{n,m}}{a_n}\rightarrow_d N(0,1).
\end{equation*}
By the definition of $\nu_{\sigma}$ (see (\ref{e20})), the
conclusion follows.
\end{proof}






\end{document}